\documentclass[10pt]{article}
\usepackage{amsfonts,amsmath,amsthm,amssymb,fullpage}
\usepackage{easybmat}

\numberwithin{equation}{section}

\newtheorem*{question*}{Question}
\newtheorem*{thm*}{Theorem}

\newtheorem{thm}{Theorem}[section]
\newtheorem{prop}[thm]{Proposition}
\newtheorem{question}[thm]{Question}

\newtheorem{cor}[thm]{Corollary}

\newtheorem{defin}[thm]{Definition}

\newtheorem{lemma}[thm]{Lemma}
\newtheorem{sublemma}[thm]{Sublemma}
\newtheorem{example}[thm]{Example}

\newcommand{\N}{\mathfrak N}
\newcommand{\Z}{\mathfrak Z}
\newcommand{\V}{\mathcal V}

\newcommand{\g}{\mathfrak g}
\newcommand{\h}{\mathfrak h}
\newcommand{\lk}{\mathfrak k}
\newcommand{\p}{\mathfrak p}

\newcommand{\restrictto}[2]{\left. #1 \right|_{#2}}

\begin{document}
\title{ Detecting orbits along subvarieties via the moment map}

\date{}
\author{M. Jablonski}
\maketitle

This work addresses the following question.  Let $G$ be a (real or complex) linear algebraic reductive group acting on an affine variety $V$ and let $W$ be a subvariety of interest.

\begin{question*}Can we calculate the size of the moduli of $G$-orbits intersecting $W$?\end{question*}

By moduli we mean the set of points up to the equivalence of lying in the same $G$ orbit.
More generally, we are interested in understanding $W$ up to $G$-equivalence; here $G$ does not preserve $W$.
We study this question when $W$ is smooth and there exists a reductive subgroup $H$ of $G$ which measures the $G$-action along $W$.  The notion of measuring the $G$-action is by means of the moment map and we say that $G$ is $H$-detected along $W$ if $m_G(w) \in \h = Lie \ H$ for $w\in W$, where $m_G$ is the moment map of the $G$ action (see Section \ref{section: prelim}).  Our main result is the following.\\

\noindent \textbf{Theorem \ref{thm: finiteness of G vs H orbits} }  \textit{  Let $G,H,V,W$ be as above.  Suppose $G$ is $H$-detected along W.  For $w\in W\subset V$,   the components of $G\cdot w \cap W$ are $H_0$-orbits, where $H_0$ is the identity component of $H$.  Consequently, $G\cdot w \cap W$ is a finite union of $H$-orbits.
}\\

As an application, we apply our work to the problem of finding continuous families of non-isomorphic  nilpotent Lie groups which do not admit left-invariant Ricci soliton metrics  (Section \ref{section: applications to left-invar}).  Additionally, this work can be applied to finding continuous families of non-isomorphic nilpotent Lie groups which do admit left-invariant Ricci soliton metrics, see \cite{Jablo:ModuliOfEinsteinAndNoneinstein}.  Another application of the above theorem to nilgeometry is the follwing.\\

\noindent \textbf{Theorem \ref{thm: N = N1+N2 and einstein condition}} \textit{  Let $N$ be a nilpotent Lie group such that $N=N_1 N_2$, a product of normal subgroups.  Then $N$ admits a left-invariant Ricci soliton metric if and only if both $N_1$ and $N_2$ admit such a metric.}\\

We finish by demonstrating our techniques applied to the adjoint representation (Section \ref{section: adjoint repn}).  Let $G$ be a semi-simple, or reductive,  group and consider the adjoint action on $\g$.  We reprove the classical result that there are finitely many nilpotent orbits and show that
every nilpotent orbit in $\g$ is (analytically) distinguished in the sense of Definition \ref{def: distinguished point/orbit}; that is, each nilpotent orbit contains a critical point of the norm squared of the moment map.

\section{Preliminaries}\label{section: prelim}

Let $G$ be a complex linear reductive group acting (rationally) on an affine algebraic variety $X$.  The following theorem is well-known.

\begin{thm} There exists a linear representation $T:G\to GL(V)$ and closed imbedding $i:X\to V$ which is $G$-equivariant; that is, $i(gx) = T(g)i(x)$ for $x\in X$, $g\in G$.\end{thm}

In this way we can reduce to the setting $X=V$, for a proof see \cite[Theorem 1.5]{PopovBook}.  Let $G$ be a complex reductive group and $K$ a maximal compact subgroup of $G$.  Let $V$ be a complex vector space on which $G$ acts linearly and rationally.  We denote this action by
    $$G\times V \to V \ \ \ \ \ (g,v)\mapsto g\cdot v$$

We may endow $V$ with a positive definite Hermitian inner product $H$.  This Hermitian inner product may be written as $H=S+iA$ where $S,A$ are real-valued symmetric,(resp.) anti-symmetric real bilinear forms.  Here $S=\frac{1}{2}\{ H+\overline H\}$, $A=\frac{1}{2i}\{H - \overline H\}$, and $A(v,w) = S(v , iw )$, where $\overline z$ denotes the complex conjugate of $z\in \mathbb C$.  Equivalently, we could start with an inner product (real bilinear, positive definite form) $S$ which is $i$-invariant and build a positive definite Hermitian form $H$ as above.

A Hermitian inner product $H(\cdot , \cdot )$ on $V$ is called \textit{$G$-compatible} if $H$ is $K$-invariant.  This definition implies that $G$ is closed under the adjoint operation.  Like wise, a (real) inner product $S$ is called $G$\textit{-compatible} if it is $K$-invariant and $i\lk$ acts symmetrically, where $\g = \lk \oplus i\lk$ and $\lk = LK$; that is, if it is $K$-invariant and $i$-invariant.
Such inner products always exist on $V$ by averaging% over $S^1\times K$
.  We will denote $S(\cdot , \cdot )$ by  $<\cdot , \cdot >$ also.

\subsection*{Cartan Involutions}
%We may regard $G$ as a real algebraic group by restriction of scalars.
Let $G$ be a closed subgroup of $GL(V)$.
%Endow $V$ with a Hermitian inner product which is invariant as above.
The adjoint with respect to $H$ is
%the same as the adjoint with respect to $S$ which we will
denoted by $*$.  Any involution of the form $\theta (g) = (g^*)^{-1}$ for $g\in GL(V)$ is called a \textit{Cartan involution} of $GL(V)$.  If the involution $\theta$ leaves $G$ stable, then we say that $\theta$ is a Cartan involution of $G$.  Such involutions always exist as we demonstrate.

Choose $H$, or equivalently $S$, to be invariant under $K$ and $i$, as above.  Then $G$ is stable under the metric adjoint operation.  This gives a Cartan involution $\theta (g) = (g^*)^{-1}$ on $G$ with $K=G^\theta$.  We denote the corresponding involution on $\g$ by $\theta$ and we have $\lk = LK = +1$ eigenspace of $\theta$ and $\p=i\lk = -1$ eigenspace of $\theta$.  We observe that $\theta$ on $\g$ is conjugate linear and is just complex conjugation with respect to the real form $\lk$ of $\g$.   See \cite{Mostow:SelfAdjointGroups}, \cite{RichSlow}, and references therein for more information on Cartan involutions and decompositions.

\begin{prop}[Mostow]\label{prop: cartan decomp of H and G} Let $G$ be a complex (linear algebraic) reductive group and $H$ a reductive subgroup.  Then there exists a Cartan involution $\theta$ which simultaneously preserves $G$ and $H$.  The involution $\theta$ can be chosen so that $K_H = H^\theta$ is a previously chosen maximal compact subgroup of $H$ and $K_H \subset K= G^\theta $ where $K$ is a maximal compact subgroup of $G$.
\end{prop}

We say that the Cartan decomposition of $H$ above is \textit{compatible} with the Cartan decomposition of $G$.  In practice we will be interested in inner products which are both $G$ and $H$-compatible. The above proposition says that such inner products always exist.  If an inner product is $G$-compatible and $H$ is compatible with $G$, then the given inner product is $H$-compatible.\\

Similarly Cartan involutions exist on real algebraic reductive groups and the above proposition is still valid.  In the real setting, however, the space $\p = -1$ eigenspace of $\theta$ will not equal $i\lk$ as $\g$ might not be a complex Lie algebra.  Regardless,  Cartan involutions and the decomposition $\g = \lk \oplus \p$ are  good tools for studying the the geometry of orbits of  real rational representations of real reductive groups, see \cite{RichSlow}, \cite{Marian}, and \cite{EberleinJablo} for applications to the study of real group orbits.

\subsection*{Moment maps}
Endow $\g$ with an $Ad\ K$-invariant, $i$-invariant, $\theta$-invariant  inner product $<<,>>$.  This is equivalent to $<<,>>$ being the real part of an $Ad\ K$-invariant, $\theta$-invariant Hermitian inner product; we choose $<<,>>$ to be $\theta$-invariant so that $g=\lk \oplus \p$ is an orthogonal decomposition.  Such inner products on $\g$ always exist as demonstrated by the following examples.  Additionally, one could guarantee their existence by averaging.

\begin{example}\label{ex: inner product on sl_n} Let $G$ be an algebraic reductive subgroup of $SL(E)$.  Let $\theta$ denote a $G$ stable Cartan involution of $SL(E)$ and let $B$ denote the Killing form of $\mathfrak{sl}(E)$.  Then the inner product $<<\cdot , \cdot >> = -B(\cdot , \theta(\cdot ))$ satisfies the conditions stated above.\end{example}

\begin{example}If $G$ is semi-simple then one may use the inner product $<<\cdot , \cdot >> = - B_\g (\cdot, \theta(\cdot ))$ where $B_\g$ is the Killing form of $\g$.\end{example}

Consider the identification/isomorphism $\varphi: \g \to  \g^* = Hom_\mathbb R(\g , \mathbb R)$
defined via $<<,>>$; that is, $\varphi (X) = <<X,\cdot >> $.  If we consider the coadjoint action of $G$ on $\g^*$ defined by $(k\cdot F)(Y) = F(k^{-1}\cdot Y)$ then the isomorphism $\varphi$ is $K$-equivariant.  We endow $\g^*$ with the inner product $<<\varphi(X),\varphi(Y)>>=<<X,Y>>$ so that $\varphi$ is an isometry.  This inner product on $\g^*$ is $K$-invariant (with respect to the coadjoint action).   Moreover, we have $\lk ^* \simeq \lk$ and $\p ^* \simeq \p$ and the decomposition $\g^* = \lk^* \oplus \p^*$ is orthogonal.

We define the $\g^*$-valued \textit{moment map} $m^*:V\to \g^*/\lk^* \simeq \p* \subset \g^*$ as follows.  First, consider the function $\rho_v:G\to \mathbb R$ defined by $\rho_v(g)=|g\cdot v|^2$, where $|u|^2 = <u,u>=S(u,u)=H(u,u)$.  The function $m^*: V\to \g^*$ is defined by
    $$m^*(v)=d(\rho_v)_e$$
The $K$-invariance of $|\cdot |^2$ implies $m^*(v)|_{\lk} = 0$.  %Hence $m^*$ takes values in $i\lk^*$ as $i\lk^*$ is orthogonal to $\lk^*$.

We define the $\g$-valued \textit{moment map} $m:V\to \g$ by $m^*=\varphi(m)$.  Equivalently, $m$ is defined implicitly by
    $$<<m(v),X>> = 2<X\cdot v, v> \ \mbox{ for all } X\in \g$$
The $K$-invariance of $<,>$ implies $m(v)\in \lk ^\perp \subset \g$.  The $\theta$-invariance of $<<,>>$ implies $\lk^\perp = \p$ and thence
    $$m(v) \in \p \ \mbox{ for all } v\in V$$
Having $\p$-valued moment maps is the reason we required $<<,>>$ to be $\theta$-invariant.\\

We observe that $m,m^*$ are (real) homogeneous polynomials of degree 2.  Moreover, for $c\in \mathbb C$ we have $m(cv)=|c|^2m(v)$ and $m^*(cv)=|c|^2m^*(v)$, where $|c|^2 = c \ \overline c$ is the usual norm square on $\mathbb C$.  This gives rise to well-defined polynomials on complex projective space
    $$m[v]=\frac{m(v)}{|v|^2} \ ,\ \ \ \ m^*[v]=\frac{m^*(v)}{|v|^2}$$
The norm of these functions is the same.  Historically more attention has been placed on $m^*$, however to study $||m||^2 = ||m^*||^2$ we will work with the function $m$ as it is very natural from
the perspective of the problems addressed in this note.

When $G$ is a real reductive group the same results hold except that $\p$ is not necessarily $i\lk$.  Moreover, in the real setting our moment maps give rise to well-defined functions on real projective space.

\subsection*{Closed and distinguished orbits}

Consider the action of $G$ on $V$.  It has been shown that the closed orbits are precisely the orbits which intersect the zero set of the moment map $m: V \to \p$.  In the complex setting this theorem was done by \cite{Kempf-Ness} and in the real setting done by \cite{RichSlow}.  We state it below.

\begin{thm} Consider the action $G\times V \to V$.  Denote the zero set of $m:V\to \p$ by $\mathfrak M = m^{-1}(0)$.  For $v\in V$, the orbit $G\cdot v$ is closed if and only if $G\cdot v \cap \mathfrak M \not = \emptyset$.
\end{thm}

The set of points whose $G$-orbits are closed is called the set of \textit{stable points}.  In contrast, a point is called \textit{unstable} if its $G$-orbit contains zero.  The set of unstable points is called the \textit{null-cone}.  This set of points can be studied in a more refined way by passing to projective space.

\begin{defin}\label{def: distinguished point/orbit} A point $v\in V$ or $[v]\in \mathbb PV$ is called $G$-distinguished, or just distinguished when the $G$ action is clear, if $[v]\in \mathbb PV$ is a critical point of $||m||^2:\mathbb PV \to \mathbb R$.  Likewise, we call an orbit $G\cdot v$, or $G\cdot [v]$, distinguished if it contains a distinguished point.
\end{defin}

Closed orbits are always distinguished as zero is an absolute minimum of $||m||^2$.  The non-closed distinguished orbits all lie in the null-cone.

\begin{lemma}\label{lemma: dist point iff m(v)v=cv} Let $v\in V$.  Then $v$ is distinguished if and only if $m(v)\cdot v = cv$ for some $c\in \mathbb C$.
\end{lemma}

We omit the proof of this lemma as it follows immediately from the definitions.

\begin{thm}\label{thm: disting orbits and -grad flow of m}  Let $G$, $\g$, $K$, and $V$ be as above endowed with  inner products as above.  Let $m$ denote the moment map of the representation and, for $v\in V$, denote by $\varphi_t(v)$ the negative gradient flow of $||m||^2$ starting at $v$.  Denote the limit point of this flow by $\varphi_\infty(v)$. If $G\cdot v$ is a distinguished orbit with $v$ such a distinguished point, then for every $g\in G$, $\varphi_\infty(g\cdot v)\in K\cdot v$.
\end{thm}

This theorem is true in both the real and complex settings, see \cite{JabloDistinguishedOrbits}.  A priori, it is not clear that the limit set of the flow is a single point.  For this point and more information on moment maps, see \cite[Section 2.5]{Sjamaar:ConvexityofMomentMapping}.

\subsection*{Detecting $G$-orbits along subvarieties}

Let $G$ be a (real or complex algebraic) reductive group which acts linearly and rationally on $V$.  Let $H$ be a reductive subgroup of $G$ which  has a compatible Cartan decomposition (see Proposition \ref{prop: cartan decomp of H and G}).  Let $V$ be endowed with a $G$-compatible metric (see beginning of Section \ref{section: prelim}).  Recall that the $G$-compatible metric on $V$ is also $H$-compatible as $H$ has a compatible Cartan decomposition.  Let $W$ be an $H$-stable smooth subvariety of $V$.

\begin{defin}\label{def: G orbit H detected along W}  We say that the $G$-action on $V$ is `$H$-detectable along $W$' if $m_G(w) \in \h$ for $w\in W$.
\end{defin}
  Similarly, we could state this definition for smooth varieties in projective space.\\

We observe the following. Let $V$ be a $G$-representation which is $H$-detectable on  $W$.  Then ${m_G(w)=m_H(w)}$ for all $w\in W$.  The proof follows immediately upon writing out the definitions, as we demonstrate.  As $m_G(w) \in \h$, we can show the desired equality by comparing  the inner products of $m_G(w)$ and $m_H(w)$ with every element of $X\in \h$
    \begin{eqnarray*}<m_G(w),X>_\g &=& <X\cdot w,w>_V \ \ \ \mbox{ from the definition of the moment map}\\
        %&=& <X\cdot w,w>_W \ \ \mbox{ as $H$ acts on $W$ and the metric on $W$ is the induced metric from $V$ }\\
        &=& <m_H(w),X>_\h \mbox{ from the definition of the moment map}\\
        &=& <m_H(w),X>_\g \mbox{ as the metric on $\h$ is the induced metric from $\g$}
    \end{eqnarray*}

\section{Orbits of Compatible Subgroups}\label{section: orbits of compatible subgroups}

Let $G$ be a reductive group with a fixed choice of Cartan decomposition relative to a Cartan involution $\theta$.  Let $H$ be a reductive subgroup which is compatible with the choice of Cartan decomposition; that is, $H$ is  $\theta$-invariant (cf. Proposition \ref{prop: cartan decomp of H and G}).  If $G$ acts on $V$ and is endowed with a $G$-compatible inner product $<,>$, then the inner product $<,>$ is $H$-compatible; that is, $H$ is self-adjoint.

Let $m_G$, resp. $m_H$, denote the moment map of $G$, resp. $H$, acting on $V$.  Suppose there exists an $H$-stable smooth subvariety $W\subset V$ on which the $G$-action is $H$-detectable; that is, such that $m_G(w)=m_H(w)$ for all $w\in W$ (see the definition  above).  The subvariety $W$ is not required to be closed.\\

The following theorems are true for real and complex groups.  We first give proofs for complex groups and finish the section by explaining how to extend the results over $\mathbb C$ to results over $\mathbb R$.

\begin{thm}\label{thm: finiteness of G vs H orbits}  Let $G,H,W,V$ be as above. If $w\in W\subset V$, then the components of $G\cdot w \cap W$ are $H_0$-orbits and, consequently, $G\cdot w \cap W$ is a finite union of $H$-orbits.
\end{thm}

\textit{Remark.} In this way we obtain a solution of the original question.  The dimension of the moduli of $G$ orbits which intersect $W$ is precisely the dimension of the moduli of $H$ orbits in $W$.  Moreover, if the $H$ action on $W$ were stable, then a lower bound on the dimension of the moduli of $H$ orbits would be the dimension of $W//H$, the GIT quotient.  See \cite{MumfordFogartyKirwan:GIT} for more information on Geometric Invariant Theory (GIT) and quotients.

\begin{cor} Let $G,H,W,V$ be as in the theorem above.  Then for $w\in W$, the intersection $G\cdot w \cap V$ is smooth.
\end{cor}

\begin{thm}\label{thm: finiteness of G vs H orbits on PV} Let $G,H,W,V$ be as above but with  $W$ a cone in $V$; that is, $W$ descends to a projective variety $\mathbb PW$ of $\mathbb PV$. Consider the induced actions on $\mathbb PV, \mathbb PW$.  If $[w]\in \mathbb PW$, then the components of $G\cdot [w] \cap \mathbb PW$ are $H_0$-orbits and, consequently, $G\cdot [w] \cap \mathbb PW$ is a finite union of $H$-orbits.
\end{thm}

\textit{Remark}.  The proof of the second theorem does not follow immediately from the first.  However, the proofs are very similar and we give them simultaneously.

\begin{cor}\label{cor: G dist iff H dist} Let $w\in W \subset V$ where $W$ is a closed $H$-stable smooth subvariety (e.g., a subspace).  Then $G\cdot w$ is distinguished if and only if $H\cdot w$ is distinguished.
\end{cor}

\textit{Remarks}.
(1) Closedness of $W$ is necessary in this corollary as finding distinguished points involves taking limits.  This corollary gives a useful criterion for determining when a specific orbit of $G$ or $H$ is distinguished.  We give several worthwhile applications of this corollary in the following section and also in \cite{Jablo:ModuliOfEinsteinAndNoneinstein}.

(2) One of our main applications is to study the negative gradient flow of the norm squared of the moment map.  A problem of interest is to understand when the limit points of this flow are contained in the orbit of the initial point.  However, Theorems \ref{thm: finiteness of G vs H orbits} and \ref{thm: finiteness of G vs H orbits on PV} can be applied to much more general settings of evolutions.  If one knows that one can evolve within a special subgroup and that one has convergence in the large group orbit, then one can achieve convergence within the subgroup orbit.  This is a `specialized change of basis' type result, cf. Section \ref{section: applications to left-invar}.

(3) It is not true, in general, that $H\cdot w$ must be distinguished if $G\cdot w$ is distinguished.  In the special case of closed orbits, this problem has been explored in \cite{Vinberg:StabilityOfReductiveGroupActions} and \cite{EberleinJablo}.  In \cite{Vinberg:StabilityOfReductiveGroupActions} it is shown that if $G\cdot w$ is closed then $H\cdot gw$ is closed for generic $g\in G$.  Other criteria to determine closedness of orbits of (sub)groups have been constructed in \cite{EberleinJablo}; at the moment there is no general criterion which completely determines when $H\cdot w$ is closed, even if $G\cdot w$ is closed.  For an explicit example of a non-closed $H$ orbit in a closed $G$ orbit, see \cite{Jablo:GoodRepresentations}.

Before proving the theorem, we use it to deduce Corollary \ref{cor: G dist iff H dist}.

\begin{proof}[Proof of the Corollary \ref{cor: G dist iff H dist}] We apply Theorem \ref{thm: finiteness of G vs H orbits on PV} and Theorem \ref{thm: disting orbits and -grad flow of m} to prove the corollary.

Recall that $[w]\in \mathbb PW$, or $w\in W$,  is a $G$-distinguished point if (by definition) $[w]$ is a critical point of $||m_G||^2 : \mathbb PV \to \mathbb R$ where $m_G$ is the moment map of the $G$-action on $\mathbb PV$.  We denote the moment map on $V$ and $\mathbb PV$ by the same notation as context should avoid any confusion.  It is well-known, see Lemma \ref{lemma: dist point iff m(v)v=cv}, that $[w]\in \mathbb PW \subset \mathbb PV$ is $G$-distinguished if and only if $m_G(w)\cdot w = c w$ for some $c\in \mathbb C$.  In this way we see that a point $w\in W$ is $H$-distinguished if and only if $w\in W$ is $G$-distinguished as $m_G(w)=m_H(w)$.  Thus, if $H\cdot w$ is $H$-distinguished, then $G\cdot w$ must be $G$-distinguished.

Conversely, consider $w\in W$ such that the orbit $G\cdot w$ is $G$-distinguished.  Theorem \ref{thm: disting orbits and -grad flow of m} provides  the existence of $g_n \in G_0$ and $g\in G$ such that $[g_n \cdot w] \to [g\cdot w]$ in $\mathbb PW$, as $n\to \infty$, and $[g\cdot w]$ is distinguished.  Notice that the sequence $[g_nw]$ can be chosen to lie in $ \mathbb PW$ by following the negative gradient flow of $||m_G||^2$ starting at $[w]$ and our limit point is in $\mathbb PW$ as $\mathbb PW$ is closed.  Moreover, observe that $[g\cdot w]$ is in the same connected component of $G\cdot [w] \cap \mathbb PW$ as $[w]$. By Theorem \ref{thm: finiteness of G vs H orbits on PV}, there exists $h\in H$ such that $g\cdot [w] = h\cdot [w]$; that is, $H\cdot [w]$ contains a distinguished point and hence $H\cdot w$ is a distinguished orbit.

\end{proof}

Next we prove Theorems \ref{thm: finiteness of G vs H orbits} and \ref{thm: finiteness of G vs H orbits on PV}. The proofs of both are given simultaneously as they are so similar.  It suffices to consider $H$ which is connected as $H$, being an algebraic group, has finitely many components.

\begin{lemma}\label{lemma: TGw cap TW = THw} For $w\in W$, $T_w (G\cdot w)\cap T_wW = T_w(H\cdot w)$.  When $W$ is a cone in $V$ with projection $\mathbb PW \subset \mathbb PV$, we have $T_{[w]}(G\cdot [w])\cap T_{[w]}\mathbb PW = T_{[w]}(H\cdot [w])$.\end{lemma}

\begin{proof}[Proof of the lemma] To show equality, we show containment in both directions.  One direction is trivial.  Since $H$ preserves $W$, we immediately have $T_w(H\cdot w) \subset T_w(G\cdot w) \cap T_w W$ which then implies $T_{[w]}(H\cdot [w]) \subset T_{[w]}(G\cdot [w])\cap T_{[w]} \mathbb PW$.  The reverse containments are shown below.

We prove the lemma at the affine level first.  By translation, we have the following identifications
    \begin{eqnarray*} T_w(G\cdot w) &\simeq& \g \cdot w \\
        T_w(H\cdot w) &\simeq& \mathfrak h \cdot w \end{eqnarray*}

Recall that our Cartan decompositions $\g = \lk \oplus \p$ and $\h = \lk_H\oplus \p_H$ satisfy $\lk_H \subset \lk$, $\p=i\lk$, and $\p_H = i\lk_H$.  Take $X\in \p \ominus \p_H$, the orthogonal compliment of $\p_H$ in $\p$, then
    $$0=<X,m_G(w)> = <X\cdot w,w> \mbox{ for all } w\in W$$
as $m_G(w)=m_H(w) \in \h$.  Consider $v_w\in T_wW$ and a curve $\gamma (t)\in W$ tangent to $v_w$ at $t=0$; that is $\gamma(0)=w$ and $\gamma'(0)=v$.  Applying the above equation to this curve and differentiating we obtain
    \begin{eqnarray*}
        0 &=& \restrictto{\frac{d}{dt}}{t=0}<X\cdot\gamma(t),\gamma(t)>\\
          &=& <X\cdot v,w>+<v,X\cdot w>\\
          &=& 2<X\cdot w,v>
          \end{eqnarray*}
as $X=X^t\in \p \ominus \p_H \subset \p$ is symmetric with respect to $<,>$.  Thus, $X\cdot w \perp T_wW$ for $w\in W$ and $X\in \p \ominus \p_H$.

Since $T_wW$ is a complex vector space and $<,>$ is $i$-invariant, for $w\in W$, $v_w\in T_wW$, and $X\in \p \ominus \p_H$ we have $0=<X\cdot w, -iv>=<iX\cdot w,v>={<(iX)\cdot w,v>} $ as   $\g$ acts $\mathbb C$-linearly.  Since $\g \ominus \h = (\lk \oplus \p)\ominus (\lk_H \oplus \p_H)=\mathbb C -span \{ \p \ominus \p_H \}$ we have
    \begin{eqnarray}\label{eqn: Xw perp W}X\cdot w \perp T_wW \mbox{ for } w\in W, \ X\in \g\ominus \h \end{eqnarray}

Now pick $X\in \g$ such that $X\cdot w \in \g \cdot w \cap T_wW$.  Writing $X=X_1 + X_2 \in \h \oplus (\g \ominus \h)$, we obtain $X_1\cdot w + X_2\cdot w \in T_wW$ which then implies $X_2\cdot w \in T_wW$ as $H$ acts on $W$ by hypothesis.  Here we are using the smoothness of $W$ to insure that $X_2\cdot w = X\cdot w-X_1\cdot w \in T_wW$.  But then $X_2\cdot w = 0$ by Equation \ref{eqn: Xw perp W}.  Hence, $X\cdot w = X_1 \cdot w \in \h \cdot w$ as desired.  This proves the lemma in the affine case.\\

To prove the second part of the lemma, we  reduce to the first part.  Recall that the map $\pi : V\to \mathbb PV$ is a submersion and the restrictions $\pi : G\cdot v \twoheadrightarrow G\cdot [v]$ and $\pi : W \twoheadrightarrow \mathbb PW$ are surjective. Let $\pi _* : TV \to T\mathbb PV$ denote the induced map on the tangent bundles and recall that  if $v_w \in T_w V$ is such that $\pi_*(v_w) = 0 \in T_{[w]} \mathbb PV$, then $v\in \mathbb C<w>$.

Consider $X_{[w]} \in T_{[w]} (G\cdot [w]) \cap T_{[w]} \mathbb PW$.  Then there exist $X'_w \in T_w(G\cdot w)$ and $X''_w \in T_wW$ such that $X_{[w]}=\pi_*(X'_w) = \pi_*(X''_w) \in T_{[w]}\mathbb PW$.  This implies $\pi_*(X'_w - X''_w) = 0 \in T_{[w]}\mathbb PW$, which implies $X'_w - X''_w \in \mathbb C<w>_w$, which then implies $X'_w \in X''_w +\mathbb C<w>_w \subset T_wW$ as $W$ is a cone and smooth at $w\neq 0$.

That is, $X'_w \in T_w (G\cdot w) \cap T_wW = T_w (H\cdot w)$ by the first part of the lemma and so $X_{[w]} = \pi_*(X'_w) \in \pi_*(T_w H\cdot w) = T_{[w]}(H\cdot [w])$.  Thus, $T_{[w]}(G\cdot [w]) \cap T_{[w]} \mathbb PW = T_{[w]}(H\cdot [w])$.

\end{proof}

As the lemma is now proven, we continue with the proof of the theorems.  If we knew $G\cdot w \cap W$ were smooth, then by dimension arguments and the lemma above it would be easy to establish that $G\cdot w \cap W \simeq H\cdot w$ locally near $w$.  A priori, however, we cannot guarantee smoothness; obviously the theorems show smoothness a posteriori.  We recall the following basic proposition on algebraic group actions; for a proof see \cite[I.1.8]{Borel:LinAlgGrps}.

\begin{prop}\label{prop: G actions on varieties}Let $G$ be an  algebraic group acting on a variety $X$.  Then for $x\in X$, the boundary $\partial (G\cdot x) = \overline{G\cdot x}-G\cdot x$ consists of $G$-orbits of strictly lesser dimension.  Moreover, the orbit $G\cdot x$ is open and dense in $\overline {G\cdot x}$.  The Hausdorff and Zariski closures of the orbit coincide. \end{prop}

\begin{lemma}\label{lemma: Jantzen}  Let $G$ be an algebraic group and let $H$ be a closed subgroup of $G$.  Let $G$ act on a variety $V$ and assume that $W$ is an $H$-stable subvariety.  Suppose the following is true for every $w\in W$
    $$ T_w (G\cdot w)\cap T_wW = T_w(H\cdot w)$$
Then the intersection of each $G$-orbit with $W$ is a union of finitely many $H$-orbits.
\end{lemma}

Clearly this lemma combined with the previous lemma completes the proof of the theorems.  We would like to thank Chuck Hague for pointing out to us that this lemma previously existed in the literature (see \cite[Lemma 2.4]{Jantzen:NilpotentOrbitsInRepresentationTheory}).\\

\textit{Proof of the lemma.}  Let $w\in W$, we can decompose the variety ${G\cdot w}\cap W = \displaystyle \cup  X_i$ into irreducible components $ X_i$.  As we are interested in the component of $G\cdot w \cap W$ that contains $w$, we only need to consider the irreducible components $X_i$ that contain $w$.\\

We begin by showing  $H$ preserves $X_i$.  Take $p_i \in X_i$ which is a smooth point of $X_i$.  Then $G\cdot p_i \cap W$ coincides with $X_i$ near $p_i$ (in the Hausdorff sense), and so $h\cdot p_i \in X_i$ for $h\in \mathcal O_e \subset H$ where $\mathcal O_e$ is a (Hausdorff) open neighborhood of the identity element $e\in H$.  Let $\mu : G \times V \to V$ denote the $G$-action on $V$.  Since $\overline{\mathcal O_e}=H$ and $\mu$ is continuous we have
    $$\mu(H,p_i) = \mu(\overline{\mathcal O_e},p_i)\subset \overline{\mu(\mathcal O_e,p_i)} \subset  X_i$$
This holds for any smooth point of $X_i$.  Let $\mathcal U$ denote the set of smooth points of $X_i$.  Then $\overline {\mathcal U}=  X_i$ (cf. \cite[Lemma 2.1.1]{Shaf2}) and as before
    $$\mu(H, X_i) = \mu(H, \overline{\mathcal U})\subset \overline{\mu(H,\mathcal U)} \subset \overline {\bigcup_{p_i\in \mathcal U}\mu(H,p_i) } \subset \overline {\bigcup_{p_i\in \mathcal U}  X_i } =
     X_i$$
This shows $H$ acts on $X_i$.

Now that $H\cdot p_i \subset X_i$ we may compare the following dimensions
    $$  \dim  H \cdot p_i \leq \dim X_i \mbox{ at } p_i \leq \dim (T_{p_i}(G\cdot p_i)\cap T_{p_i}W  ) = \dim T_{p_i}(H\cdot p_i)$$
where the second inequality follows from $T_{p_i}(G \cdot {p_i}\cap W)\subset T_{p_i}(G\cdot p_i)\cap T_{p_i}W$ and the
last equality follows from the lemma above.  Thus $\dim H \cdot p_i = \dim X_i $ and we see that $\dim H \cdot p_i $ is an (analytic) open neighborhood of $p_i$ in $X_i$.  As $ X_i$ is irreducible, we see that $\overline {H\cdot p_i} = X_i$.  Since $H\cdot p_i$ is open and dense in $\overline{H\cdot p_i}$, the same is true for $H\cdot p_i \subset X_i$.\\

Next we show $H\cdot p_i = X_i$.  Let $w'\in X_i$.  First we observe that $w'\in \overline{H\cdot p_i}$, as stated in the previous paragraph.  The next lemma will use the fact that the Hausdorff and Zariski closure of the orbit coincide (cf. Proposition \ref{prop: G actions on varieties}).

\begin{sublemma} $\dim T_{w'}(G\cdot w')\cap T_{w'}W \geq \dim H\cdot p_i$
\end{sublemma}

\begin{proof}[Proof of the sublemma] By hypothesis $G\cdot w' = G\cdot p_i$ and there exists $w_n\in H\cdot p_i$ such that $w_n \to w'$ as $w'$ is in the Hausdorff closure of $H\cdot p_i$.

Pick an orthonormal basis $\{ X_j^n \}_{j=1}^r$ of $T_{w_n}H\cdot w_n = T_{w_n}H\cdot p_i$ where $r=\dim H\cdot p_i$.  By passing to a subsequence we may assume $X_j^n\to X_j$ as $n\to \infty$, for $j=1,\dots , r$.  Since $G\cdot w'=G\cdot p_i$ we have $X_j \in T_{w'} G\cdot w'$ which then implies $\{ X_j\}_{j=1}^r \subset T_{w'}(G\cdot w') \cap T_{w'}W$.  As this collection of vectors is orthonormal, the sublemma is proven.

\end{proof}

By hypothesis and the above work, we have
    $$\dim H\cdot w' = \dim T_{w'}(H\cdot w') = \dim(T_{w'}G\cdot w' \cap T_{w'}W) \geq \dim H\cdot p_i$$
But $w'\in \overline{H\cdot p_i}$.  Proposition \ref{prop: G actions on varieties} implies $\dim H\cdot w' \leq \dim H\cdot p_i$ with equality if and only if $H\cdot w'=H\cdot p_i$. Applying this fact
and the inequality above, we see that $H\cdot w'=H\cdot p_i$.  Since $w'\in X_i$ was arbitrary we have $H\cdot p_i =X_i$.  We observe that we have shown $X_i = H\cdot w'$ for any $w'\in X_i$.\\

Recall that we have decomposed $G\cdot w \cap W = \cup X_i$ into irreducible components.  If an irreducible component $X_i$ contains $w$ then $X_i = H\cdot w$, as shown above. Thus the topological component of $G\cdot w \cap W$ containing $w$ is $\displaystyle \bigcup_{ \{X_i| \ w\in X_i\} } X_i = \bigcup_{} H\cdot w = H\cdot w$.  This proves the lemma and the proofs of our theorems are complete.\\

\textit{Remark.}  We point out that this technique of using compatible subgroups to study $G\cdot w \cap W$ does not completely detect the phenomenon of the intersection being a finite union of $H$-orbits.  There do exist examples where $W$ is an $H$-stable subspace, each intersection $G\cdot w \cap W$ is finite union of $H$-orbits, but that the $G$-action is not $H$-detectable along $W$ for any choice of inner products on $V$ and $\g$.  This is proven by constructing $H$ and $W$ so that the $G$-orbit through a generic point of $W$ is closed, but so that the $H$-orbit through said points is not closed (cf. Corollary \ref{cor: G dist iff H dist}).

\subsection*{The question of closed orbits of subgroups}

Let $G$ be a reductive group acting rationally on $V$.  Let $H$ be a reductive subgroup.  If $G\cdot v$ is closed, then it is known that $H\cdot gv$ is closed for generic $g\in G$.  This problem has been worked on by many people, see, e.g., \cite{Luna:ClosedOrbitsofReductiveGroups}, \cite{Nisnevich:StabilityAndIntersections}, \cite{Vinberg:StabilityOfReductiveGroupActions}, or \cite{Jablo:GoodRepresentations}.  However, it is not true for all $g\in G$ that $H\cdot gv$ must be closed.  It is an interesting problem to try and determine when the orbit of $H$ is closed.

\begin{cor} Suppose there exists an $H$-stable smooth (closed) subvariety $W$ along which $G$ is $H$-detected.  If $G\cdot w$ is closed, then so is $H\cdot w$.\end{cor}

This is the special case of Corollary \ref{cor: G dist iff H dist} when our distinguished orbit is closed.  We ask the following question.

\begin{question} Let $H$ be a reductive subgroup of $G$.  Let $G$ act rationally on $V$ an suppose that $G\cdot v$ and $H\cdot v$ are closed.  Do there exist inner products on $V$ and $\g$, satisfying the hypothesis of Theorem \ref{thm: finiteness of G vs H orbits}, such that $G$ is $H$-detected along the smooth subvariety $H\cdot v $?\end{question}

To the contrary we could ask

\begin{question} Does there exist $v\in V$ such that $G\cdot v$ is closed, $G$ is $H$-detected along $H\cdot v$, but $H\cdot v$ is not closed?\end{question}

As stated at the end of  the previous subsection, there do exist examples of $V,W,G,H$ such that $G$ cannot be $H$-detected along $W$ for any choice of inner products on $V$ and $\g$.

\subsection*{Real algebraic groups}

Here we explain how to obtain the above results over $\mathbb R$.  Let $G$ be a real algebraic reductive group.  The real group $G$ can be realized as the real points of a complex algebraic reductive group $G^\mathbb C$ such that $G$ is Zariski dense.  This is well-known and the construction of such a group $G^\mathbb C$ can be found in \cite{JabloDistinguishedOrbits}, for example.

The following result of Borel-Harish-Chandra (\cite[Proposition 2.3]{BHC}) is the standard way of relating the real and complex settings.  Let $V$ be real vector space on which $G$ acts linearly and rationally.  Let $V^\mathbb C = V\otimes \mathbb C$ denote the complexification, then $G^\mathbb C$ acts linearly and rationally on $V^\mathbb C$.

\begin{thm}\label{thm: BHC}  Consider $v\in V \subset V^\mathbb C$.  Then $G^\mathbb C \cdot v \cap V$ is a finite union of $G$-orbits. Moreover, the orbit $G^\mathbb C\cdot v$ is closed if and only if $G\cdot v$ is closed.
\end{thm}

In the last assertion, the only if direction requires the work of either \cite{Birkes:OrbitsofLinAlgGrps} or \cite{RichSlow}.  In fact, a slightly stronger version of this theorem is true.  We state this version below and refer the reader to \cite{JabloDistinguishedOrbits} for a proof.

\begin{thm}\label{thm: real dist iff compex dist} Consider $v\in V\subset V^\mathbb C$.  Then $G^\mathbb C \cdot v$ is distinguished if and only if $G\cdot v$ is distinguished.
\end{thm}

To obtain Theorems \ref{thm: finiteness of G vs H orbits}  and \ref{thm: finiteness of G vs H orbits on PV} over $\mathbb R$, one just needs know that they are true over $\mathbb C$ and apply Theorem \ref{thm: BHC}.  To obtain Corollary \ref{cor: G dist iff H dist} over $\mathbb R$, one just needs to know that it is true over $\mathbb C$ and apply Theorem \ref{thm: real dist iff compex dist}.

\section{Applications to the Left-Invariant Geometry of Lie Groups}\label{section: applications to left-invar}

The theorems in this section can be viewed as specialized change of basis theorems.  We are interested in left-invariant Ricci soliton metrics on nilpotent Lie groups.  Such a metric is called a \textit{nilsoliton} and if a nilpotent Lie group admits such a metric, it is called an \textit{Einstein nilradical} (see \cite{LauretWill:EinsteinSolvExistandNonexist} for justification of this terminology).  %Our approach follows J. Lauret (see section 3 of \cite{Lauret:CanonicalCompatibleMetric}).

\begin{question*} Which nilpotent Lie groups are Einstein nilradicals?  If a nilpotent Lie group admits such a metric, how can one find this special metric?
\end{question*}

As a left-invariant metric on a Lie group is equivalent to an inner product on its Lie algebra, we reduce to studying Lie algebras with inner products.  The nilsoliton condition can be completely phrased at the algebra level.

There are two points of view that one can take. The first is to fix a Lie bracket on a vector space and vary the inner product, the second is to fix an inner product on a vector space and vary the Lie bracket.  We will work from the second perspective as it has produced many results and allows us to exploit tools from Geometric Invariant Theory.  This is the perspective taken by J. Lauret and others (see \cite{Lauret:CanonicalCompatibleMetric} and references therein).

These special metrics are realized as critical points of the norm squared of the moment map corresponding to a particular $GL(n,\mathbb R)$ action. Let $\N$ be an $n$-dimensional real vector space with fixed inner product $<,>$.  Consider the space $V= \wedge^2 (\N)^* \otimes \N$.  This is the space of skew-symmetric bilinear forms from $\N \times \N$ to $\N$.  The set of Lie brackets on $\N$ is a variety $\V$ in $V$.  The change of basis action of $GL(n,\mathbb R)$ on $\N$ induces the following action on $V$
    $$g\cdot \mu (X,Y) = g\mu(g^{-1}X,g^{-1}Y)$$
for $g\in GL(n,\mathbb R)$, $\mu \in V$, and $X,Y\in \mathbb R^n$.  This action preserves the variety $\V$ of Lie brackets.  Moreover, if $\mu \in \V$ is a Lie bracket on $\N$, then the orbit $GL_n \cdot \mu$ is precisely the isomorphism class of $\mu$ in $\V$.  The metric nilpotent Lie algebra with bracket $\mu$ and inner product $<,>$ is denoted $\N_\mu$ and the simply connected nilpotent Lie group with Lie algebra $\N_\mu$ and left-invariant metric corresponding to $<,>$ is denoted by $\{ N_\mu,<,> \}$ or just $N_\mu$ when $<,>$ is understood.  Let $g*<\cdot ,\cdot > = <g^{-1}\cdot , g^{-1}\cdot >$ for $g\in GL(n,\mathbb R)$, then $\{ N_{g\cdot \mu} , <,> \}$ is isometric to $\{ N_\mu , g*<,>\}$.  Thus $N_\mu$ admits a nilsoliton metric if and only if $N_{g\cdot \mu}$ is a nilsoliton for some $g\in GL(n,\mathbb R)$.  In this way we can study the set of left-invariant metrics on $N_\mu$ by studying the orbit $g\cdot \mu$ in $\mathcal V$.\\

The inner product on $\N$ extends naturally to an inner product on $V$ defined by $<\lambda, \mu> =  \sum_{ijk} <\lambda(X_i,X_j), X_k> <\mu(X_i,X_j),X_k> $, where  $\{ X_i \}$ is  any orthonormal basis of $\N$.  This inner product on $V$ is $O(n)$-invariant, where $O(n)$ is the orthogonal group relative to $<,>$ on $\N$.

A Cartan decomposition of $\mathfrak{gl}(n)$ is given by $\mathfrak{so}(n) \oplus symm(n)$ where $\mathfrak{so}(n)$ is the Lie algebra of $O(n)$, the skew-symmetric endomorphisms relative to $<,>$, and $symm(n)$ is the set of symmetric endomorphisms relative to $<,>$.  We consider the inner product $<<A,B>> = tr (AB^t)$ on $\mathfrak{gl}(n)$, where the transpose is determined by $<,>$.  This gives rise to the following moment map (see \cite[Proposition 3.5]{Lauret:CanonicalCompatibleMetric})
    \begin{equation} m(\mu) = -4 \sum_i (ad_\mu X_i)^t ad_\mu X_i + 2\sum_i ad_\mu X_i (ad_\mu X_i)^t  \end{equation}
where $\{ X_i \}$ is any orthonormal basis of $\N$.  Or equivalently, for $X,Y \in \N$,
    \begin{equation}\label{eqn: lauret on m} < m(\mu)X,Y> = -4 \sum_{ij}<\mu(X,X_i),X_j> <\mu(Y,X_i),X_j> + 2\sum_{ij}<\mu(X_i,X_j),X><\mu(X_i,X_j),Y>\end{equation}

\begin{thm}[Lauret]\label{thm: Einstein nil equiv distinguished orbit}  The Lie group $N_\mu$ (with left-invariant metric $<,>$) is a left-invariant Ricci soliton if and only if $[\mu]$ is a critical point of $||m||^2 : \mathbb PV \to \mathbb R$; here $[\mu]$ denotes the class of $\mu$ in $\mathbb PV$.  Equivalently, $N_\mu$ is an Einstein nilradical if and only if the orbit $GL(n,\mathbb R)\cdot \mu$ is distinguished.
\end{thm}

Given $\mu \in \V$, we are interested in the problem of finding $g\in GL(n,\mathbb R)$ such that $N_{g\cdot \mu}$ is a nilsoliton.  The following theorems say that $g$ can be chosen from a subgroup of $GL(n,\mathbb R)$ that reflects natural symmetries in the metric algebra $\N_\mu$.  In this way, the following are considered specialized change of basis theorems.

\begin{thm}\label{thm: soliton via group for 1 symm der} Let $\{N, <,>\}$ be a nilpotent Lie group with left-invariant metric $<,>$.  Denote by $\{ \N, <,>\}$ the Lie algebra of $N$ with inner product $<,>$ corresponding to the left-invariant metric on $N$.

Suppose $N$ admits a left-invariant Ricci soliton metric; that is, there exists $g\in GL(\N)$ such that ${g^*<\cdot ,\cdot >=<g^{-1}\cdot , g^{-1}\cdot >}$ is a nilsoliton.  Moreover, suppose that $\{ \N, <,> \}$ admits a symmetric derivation $D\in Der(\N)$ (D is symmetric with respect to $<,>$).

Let $\N = \oplus \N_\lambda$ denote the eigenspace decomposition of $D$.  Then the element $g\in GL(\N)$ such that $g^*<,>$ is nilsoliton can actually be chosen from the subgroup $GL(\N_{\lambda_1})\times \dots \times GL(\N_{\lambda_k})$.
\end{thm}

\textit{Remark}. This theorem was known in the special case that the symmetric derivation $D$ is the unique derivation such that the rank 1 extension $\mathfrak{s} = <D>\oplus \N$ admits a left-invariant Einstein metric \cite[Proposition 6.8]{Heber}.  In fact, there it is shown that a slightly smaller group can be used.  However, our theorem is very useful in practice when it is not known which symmetric derivation should be used to uniquely extend to an Einstein solvmanifold.\\

\begin{cor} Let $\{N, <,>\}$ be a two-step nilpotent Lie group with Lie algebra $\N$.  Denote the center of $\N$ by $\Z$.  Then $N$ is an Einstein nilradical if and only if there exist $g\in GL(\Z^\perp)\times GL(\Z)$ such that $\{ N, g*<,> \}$ is a nilsoliton; here $\Z^\perp \subset \N$ is taken relative to $<,>$.
\end{cor}

\begin{proof}[Proof of Corollary]
This corollary follows immediately from the theorem as every two-step nilpotent Lie algebra admits a symmetric derivation defined by $Id$ on $\Z^\perp$ and $2 Id$ on $\Z$.
\end{proof}

\begin{proof}[Proof of Theorem \ref{thm: soliton via group for 1 symm der}]  Let $N=N_\mu$ for some $\mu \in \V$.  We will apply Corollary  \ref{cor: G dist iff H dist} for the particular representation at hand.

By hypothesis, our nilpotent Lie algebra is an Einstein nilradical and so Theorem \ref{thm: Einstein nil equiv distinguished orbit} implies that the $G=GL(n,\mathbb R)$ orbit is distinguished.  Consider the subspace $W = \{ \lambda \in V \ | \ D\lambda(X,Y) = \lambda(DX,Y) + \lambda(X,DY) \mbox{ for } X,Y\in \N \}$.  This is a vector subspace which contains $\mu$ as $D\in Der (\N_\mu)$.  \\

\textbf{Lemma.}  Let $\alpha, \beta$ be eigenvalues $D$ with corresponding eigenspaces $V_\alpha, V_\beta$, then for $\lambda \in W$,  $\lambda(V_\alpha,V_\beta) \subset V_{\alpha+\beta}$, the eigenspace corresponding to $\alpha + \beta$.\\

The proof is immediate.\\

Define $H$ to be the group $GL(\N_{\lambda_1}) \times \dots \times GL(V_{\lambda_k})$.    First observe that $H$ is closed under the metric adjoint with respect to $<,>$, hence $H$ is reductive.

For  $\lambda \in W$, we will show $m(\lambda) \in \h = LH$ by means of Equation \ref{eqn: lauret on m}.  Let $X_{\alpha}$ (resp. $X_{\beta}$) $ \in \N$ be in the $\lambda_\alpha$ (resp. $\lambda_\beta$) eigenspace of $D$, $\lambda_\alpha \not = \lambda_\beta$.  Choose an orthonormal basis $\{ X_i \}$ consisting of eigenvectors of $D$, and apply Equation \ref{eqn: lauret on m}.  This gives
    \begin{eqnarray*} <m(\lambda) X_\alpha, X_\beta> &=& -4 \sum_{ij}<\lambda(X_\alpha,X_i),X_j> <\lambda(X_\beta,X_i),X_j> \\ && + 2\sum_{ij}<\lambda(X_i,X_j),X_\alpha><\lambda(X_i,X_j),X_\beta>\\
    &=& 0
    \end{eqnarray*}
To see that this is zero, we observe that each summand is zero.  In the first summation, we have $<\mu(X_\alpha,X_i),X_j> <\mu(X_\beta,X_i),X_j> = 0$ since either $\lambda_\alpha + \lambda_i \not = \lambda_j$ or $\lambda_\beta + \lambda_i \not = \lambda_j$ as $\lambda_\alpha\not=\lambda_\beta$; here we are applying  the lemma above.  Similarly, all the terms in the second summation are zero.  Thus $m(\lambda) \in \h$ for $\lambda \in W$.

Applying Corollary \ref{cor: G dist iff H dist} together with Theorem \ref{thm: Einstein nil equiv distinguished orbit} completes the proof.
\end{proof}

\begin{thm} Consider $\{ N, <,> \}$ and a collection of symmetric derivations $\{ D_\alpha \}$ of $\N$ (relative to $<,>$).  Denote by $H^\alpha$ the group which preserves the eigenspace decomposition of $D_\alpha$, as in Theorem \ref{thm: soliton via group for 1 symm der}.  Define $H = \cap H^\alpha$.  The group $H$ is a reductive algebraic group and $N$ admits a nilsoliton metric if and only if $\{N, h*<,> \}$ is such a metric for some $h\in H$.
\end{thm}

\textit{Remark}.  It is not always true that the intersection of reductive groups is reductive, although such intersections are `generically' reductive.  See,  e.g., \cite{Jablo:GoodRepresentations} for an explicit example of this phenomenon.

\begin{proof}  It was shown in the proof of the previous theorem that each $H^\alpha$ is closed under the metric adjoint.  Hence, the same is true for the intersection $H$.  Moreover, $H$ is a variety as it is the intersection of varieties.  Thus, $H$ is a reductive algebraic group.  Let $\h$ be the Lie algebra of $H$, then $\h = \cap \h^\alpha$, where $\h^\alpha$ is the Lie algebra of $H^\alpha$.

Define a subspace $W_\alpha = \{ \lambda \in W \ | \ D_\alpha \lambda (X,Y) = \lambda(D_\alpha X,Y) + \lambda(X,D_\alpha Y) \mbox{ for } X,Y\in \N \}$.  Define $W = \cap W_\alpha$.  The proof of the previous theorem shows that for $\lambda \in W \subset W_\alpha$ we have $m(\lambda) \in \h^\alpha$.  Hence for $\lambda\in W$, $m(\lambda)\in \h$.

As before, applying Corollary \ref{cor: G dist iff H dist} together with Theorem \ref{thm: Einstein nil equiv distinguished orbit} completes the proof.
\end{proof}

\begin{thm}\label{thm: N = N1+N2 and einstein condition} Suppose that $\N$ is a nilpotent algebra that can be written as a sum of ideals $\N = \N_1 \oplus \N_2$.  Then $N$ is an Einstein nilradical if and only if both $\N_1, \N_2$ are Einstein nilradicals.
\end{thm}

\textit{Remark.}  The property of being an Einstein nilradical is a property of the Lie group (or algebra) and does not depend on the choice of inner product on $\N$.

\begin{proof} We choose to endow $\N$ with an inner product so that $\N_1$ and $\N_2$ are orthogonal.  With this choice of inner product we have $\N = \N_\mu = \N_{\mu_1}\oplus \N_{\mu_2}$ where $\mu = \mu_1 + \mu_2 \in W:= (\wedge ^2 \N_1^* \otimes \N_1) \oplus (\wedge ^2 \N_2^* \otimes \N_2) \subset \wedge ^2 \N^*\otimes \N$.  As in the previous theorems, we will apply Corollary \ref{cor: G dist iff H dist} together with Theorem \ref{thm: Einstein nil equiv distinguished orbit}.

We will show for $\lambda \in W$ that $m(\lambda)\in \mathfrak{gl}(\N_1)\times \mathfrak{gl}(\N_2)$.  We apply Equation \ref{eqn: lauret on m} by choosing an orthonormal basis which respects $\N_1\oplus \N_2$. For $X\in \N_1$, $Y\in \N_2$ we have
    $$<m(\lambda)X,Y> = -4 \sum_{ij}<\lambda(X,X_i),X_j> <\lambda(Y,X_i),X_j>  + 2\sum_{ij}<\lambda(X_i,X_j),X><\lambda(X_i,X_j),Y>$$
In the first summand, if $X_i\in \N_1$, then $\lambda(Y,X_i)=0$; if $X_i\in \N_2$, then $\lambda(X,X_i)=0$.  In the second summand, if $X_i,X_j$ are not in the same subspace $\N_1$ or $\N_2$, then $\lambda(X_i,X_j)=0$.  Now suppose, without loss of generality, that $X_i,X_j\in \N_1$.  Then we have $\lambda(X_i,X_j)\in\N_1$ with is orthogonal to $\N_2$, hence $<\lambda(X_i,X_j),Y>=0$.  In this way, we see that $<m(\lambda)X,Y>=0$ for $X\in \N_1$, $Y\in \N_2$.  Thus $m(\lambda)\in \mathfrak{gl}(\N_1)\times \mathfrak{gl}(\N_2)$ and the proof is complete.
\end{proof}

\begin{prop}  Consider $\N=\N_1\oplus \N_2$.  Let $\V_i$ denote the variety of Lie brackets on $\N_i$ (see the remarks at the beginning of this section for more information on the set of Lie brackets).  For $\mu_i\in \V_i$, consider the set of Lie brackets on $\N$ that can be written in the form $\mu = \mu_1 \oplus \mu_2$.  The moduli of such is precisely the moduli of $GL(\N_1) \times GL(\N_2)$ orbits in $\V_1 \times \V_2 \subset \V$.
\end{prop}

\begin{proof} This follows immediately from the proof of the previous theorem and Theorem \ref{thm: finiteness of G vs H orbits}, see also the remarks following  Theorem \ref{thm: finiteness of G vs H orbits}.
\end{proof}

\textit{Remark.} In this way we can construct moduli of Einstein and non-Einstein nilradicals.  For example, pick one non-Einstein nilradical $\N_1$ (which are known to exist, see \cite{LauretWill:EinsteinSolvExistandNonexist} or \cite{Jablo:ModuliOfEinsteinAndNoneinstein}) and add on the vector space $\N_2$.  Letting the bracket vary on $\N_2$ will produce moduli of algebras $\N_1\oplus \N_2$ which cannot be Einstein nilradical by the previous proposition.  However, the moduli of non-Einstein nilradicals constructed in this way is somewhat trivial.  In \cite{Jablo:ModuliOfEinsteinAndNoneinstein}, there are given constructions of non-Einstein nilradicals which do not arise as direct sums, that is, the examples given are indecomposable.

These are not the first examples of moduli of non-Einstein nilradicals.  In \cite{Will:CurveOfNonEinsteinNilradicals}, C. Will produces a curve of (pairwise) non-isomorphic, non-Einstein nilradicals.  To our knowledge, these are the only other known examples of moduli of this phenomenon.\\

%\textcolor{red}{Find out about Cynthia's example}.\\

In the study of two-step nilpotent Lie algebras there are different representations that are very useful for obtaining concrete results.  We will not go into the details of setting these up and refer the reader to \cite{Jablo:ModuliOfEinsteinAndNoneinstein} for the applications of this work in the two-step case.  The work of this paper allows one to concretely construct moduli of Einstein and non-Einstein nilradicals in the set of ``non-generic" two-step nilalgebras.

Although we do not present the details here, one interesting example that comes from the two-step setting is the following.  See the two-step case of type $(2,2k+1)$ in \cite{Jablo:ModuliOfEinsteinAndNoneinstein}.

\begin{example} There exist representations with one generic orbit, but with moduli (of dimension $\geq 1$) of non-generic orbits in the null-cone. \end{example}

\section{Nilpotent Orbits in the Adjoint Representation}\label{section: adjoint repn}

In this section we present the case of the adjoint representation.  We focus our attention on the null-cone of this representation.  Recall that the null-cone consists of a representation $V$ consists of points $v\in V$ such that $0\in \overline{G\cdot v}$.  When $V=\g$ is the adjoint representation, $X\in \g$ is in the null-cone if and only if $X$ is nilpotent, see, e.g., Remark 10.2 of \cite{BHC}.

It is well-known that there exist only finitely many orbits in the null-cone of an adjoint representation of a reductive group.  We reprove this result by applying Theorem \ref{thm: finiteness of G vs H orbits}.  Moreover, we show that each of these orbits is distinguished (in the sense of Definition \ref{def: distinguished point/orbit}).

The known general proofs of the finiteness result are essentially the same as ours, without the language of the moment map.  These proofs use
Lemma \ref{lemma: Jantzen} and an additional theorem on the reducibility of representations of reductive groups.  These ideas originally go back to Richardson \cite{Richardson:ConjClassesinLieAlgebrasandAlgebraicGroups}.

\begin{prop}\label{prop: SL_n finite nullcone} Let $G = SL_n(\mathbb C)$.  The null-cone of the adjoint action of $G$ on $\g$ has finitely many orbits.
\end{prop}

These orbits are in one-to-one correspondence with the partitions of $n$; they correspond to the different Jordan normal forms with $0's$ along the diagonal (the nilpotent endomorphisms).  See \cite[Section 1.1]{Jantzen:NilpotentOrbitsInRepresentationTheory} for more details.

\begin{thm}[Dynkin-Kostant] Let $G$ be a (real or complex algebraic) linear reductive group.  There exist only finitely many orbits in the null-cone of the adjoint representation. \end{thm}

We prove this theorem in the complex setting and apply Theorem \ref{thm: BHC} to immediately obtain the theorem for reals groups once known for complex groups (cf. the end of Section \ref{section: orbits of compatible subgroups}).

\begin{proof}  It suffices to consider connected reductive groups $G$. As $G$ is a linear group, $G \subset SL_n(\mathbb C)$ for some $n\in \mathbb N$.  Let $\theta$ be a $G$-stable Cartan involution of $SL_n(\mathbb C)$, such exists by Proposition \ref{prop: cartan decomp of H and G}.  Then $\g \subset \mathfrak{sl}_n(\mathbb C)$ and we will endow $\mathfrak{sl}_n(\mathbb C)$ with the inner product from Example \ref{ex: inner product on sl_n}.  That is, $< \cdot,\cdot> = -B(\cdot, \theta(\cdot))$ where $B$ is the Killing form of $\mathfrak{sl}_n(\mathbb C)$.  This inner product is $SU(n), K,$ and $i$-invariant; and so, this inner product satisfies the requirements of Section \ref{section: prelim}.  We endow $\mathfrak{sl}_n(\mathbb C) = L \ SL_n(\mathbb C)$ and $\g = LG$ with the same inner product inner product $<<,>>=<,>$.

From this one readily computes
    \begin{equation}\label{eqn: m(v) for adjoint repn} m(X)=-[X,\theta(X)]\end{equation}
for $X\in \mathfrak{sl}_n(\mathbb C)$.  Observe that the $SL_n(\mathbb C)$ orbits are $G$-detected along $\g$ (see Definition \ref{def: G orbit H detected along W}); that is, the moment map takes values in $\g$ when evaluated along $\g$.

We also observe that the $G$-null-cone in $\g$  is contained in the $SL_n(\mathbb C)$-null-cone in $\mathfrak{sl}_n(\mathbb C)$.  Applying Theorem \ref{thm: finiteness of G vs H orbits} we see that each $SL_n(\mathbb C)$ orbit through the $G$-null-cone consists of finitely many $G$-orbits.  However, there are only finitely many such $SL_n(\mathbb C)$-orbits (see Proposition \ref{prop: SL_n finite nullcone}).  Hence, there are only finitely many $G$-orbits in the $G$-null-cone.
\end{proof}

\begin{thm} Let $G$ be a semi-simple group acting on $\g$ by the adjoint representation.  Every orbit in the null-cone of $\g$ is distinguished (in the sense of Definition \ref{def: distinguished point/orbit}).
\end{thm}

\textit{Remark.}  This result has previously appeared in the literature, see Lemma 2.11 of \cite{SchmidVilonen:OnTheGeometryofNilpotentOrbits}.  The proof presented in that work uses the work of \cite{Sekiguchi:RemarksOnRealNilpotentOrbitsofaSymmetricPair}.  Our proof relies on constructing such critical points for the case of $\mathfrak {sl}_n(\mathbb C)$ and then applying Corollary \ref{cor: G dist iff H dist}.

\begin{prop}  Every nilpotent orbit in the adjoint representation of $SL_n$ is distinguished.
\end{prop}

\begin{proof} We first prove that the principal orbit in $\mathfrak{sl}_n$ is distinguished.  Consider the element
    $$ J_n = \begin{bmatrix}
            0&\lambda_1\\
            &0&\lambda_2\\
            &&\ddots &\ddots \\
            &&&0&\lambda_{n-1}\\
            &&&&0
            \end{bmatrix}$$
We point out that this element is in the same $SL_n$-orbit as the standard nilpotent Jordan block as long as the $\lambda_i$ are non-zero.

By Equation \ref{eqn: m(v) for adjoint repn}, $m(X) = -[X, -X^*] = XX^*-X^*X$ and
    $$m(X)\cdot X = [m(X),X]=m(X)X-Xm(X)= 2XX^*X - X^*X^2-X^2X^*$$
We claim that $\lambda_i = \sqrt{\frac{i}{2}(n-i)}$ satisfies the condition $m(J_n)\cdot J_n =J_n$ (cf. Definition \ref{def: distinguished point/orbit} and Lemma thereafter with $c=1$) and hence $SL_n\cdot J_n$ will be distinguished.  This is easy to verify by direct computation.

To see that every nilpotent orbit contains a distinguished point, observe that every nilpotent endomorphism can be conjugated to be of the form
    $$X = \begin{bmatrix} J_{k_1}\\
        & J_{k_2}\\
        &&\ddots \\
        &&& J_{k_j}\end{bmatrix}$$
where $\sum k_i=n$.  This is the Jordan form where the super diagonal has non-zero entries which are rescaled.  The matrix $m(X)$ will have the same block decomposition and hence
    $$ m(X)\cdot X= \begin{bmatrix} m(J_{k_1})\cdot J_{k_1}\\
        & m(J_{k_2})\cdot J_{k_2}\\
        &&\ddots \\
        &&& m(J_{k_j})\cdot J_{k_j}\end{bmatrix} = \begin{bmatrix} J_{k_1}\\
        & J_{k_2}\\
        &&\ddots \\
        &&& J_{k_j}\end{bmatrix} = X$$
\end{proof}

We continue now with the proof of the theorem for all semi-simple groups.  We use the following observation without proof.

\begin{lemma} Let $X\in \g$ be in the null-cone.  Then $cX\in G\cdot X$ for all $c\in \mathbb K$, where $\mathbb K = \mathbb R$, respectively $\mathbb C$, if $\g$ is a real, respectively complex, Lie algebra.
\end{lemma}

First, we reduce to the case that $G$ is simple.  Suppose that the result is known for simple groups and decompose $\g= \g_1\times \dots \g_k$ as a product of simple factors.  Let $X$ be a nilpotent element of $\g$, then $X = \sum X_i$ where $X_i \in \g_i$ are all nilpotent elements.  Observe that $m(X) =\sum m(X_i)$.  Using the lemma above and the hypothesis, we may assume that each $X_i$ satisfies $m(X_i)\cdot X_i=X_i$. But now $m(X)\cdot X = \sum m(X_i)\cdot X_i = \sum X_i = X$.

All that remains to be shown is that the theorem is true in the case $G$ is simple.  Firstly, we embed $G$ into $SL_n$ so that $G$ is self-adjoint, as in the beginning of the section.  As $G$ is simple, any $K$-invariant inner product on $\p$ is unique up to scaling.  This is due to the fact that $K$ acts irreducibly on $\p$.  Thus, given any inner product on $\g$, as in Section \ref{section: prelim}, our moment map is uniquely defined up to rescaling.  As rescaling the moment map does not affect the property of a point being distinguished, we may assume that our inner product on $\g$ is the restriction of the inner product on $\mathfrak{sl}_n$.

Recall that the $SL_n$-orbits are $G$-detected along $\g$.  As every nilpotent element of $\g$ is a nilpotent element of $\mathfrak{sl}_n$ and every nilpotent orbit of $SL_n$ is $SL_n$-distinguished, applying Corollary \ref{cor: G dist iff H dist} we see that every nilpotent orbit of $G$ is $G$-distinguished.

%
%\textcolor{red}{
%\begin{question}[PERSONAL] This proof works equally well for semi-simple groups $H<G$. Then one could study the $H$-null-cone from the $G$-null-cone.  What are $H$-orbits in the $H$-null-cone which are $G$-related? That is, that lie on the same $G$-orbit.  The principal one is alone, but others probably aren't...
%\end{question}
%}

\providecommand{\bysame}{\leavevmode\hbox to3em{\hrulefill}\thinspace}
\providecommand{\MR}{\relax\ifhmode\unskip\space\fi MR }
\providecommand{\MRhref}[2]{%
  \href{http://www.ams.org/mathscinet-getitem?mr=#1}{#2}
}
\providecommand{\href}[2]{#2}

\end{document}